\documentclass[a4paper, reqno, oneside]{amsart}

\usepackage[latin1]{inputenc}
\usepackage{amsmath}
\usepackage{amsfonts}
\usepackage{amssymb}
\usepackage{graphicx}
\usepackage{enumerate}

\newtheorem{theorem}{Theorem}[section]

\theoremstyle{plain}

\numberwithin{subcase}{case}

\newtheorem{corollary}[theorem]{Corollary}

\newtheorem{lemma}[theorem]{Lemma}

\begin{document}

\author{Karl Heuer}
\address{Karl Heuer, Fachbereich Mathematik, Universit\"{a}t Hamburg, Bundesstra{\ss}e 55, 20146 Hamburg, Germany}

\title[]{Excluding a full grid minor}

\begin{abstract}
In this paper we characterise the graphs containing a $\mathbb{Z} \times \mathbb{Z}$ grid minor in a similar way as it has been done by Halin for graphs with an $\mathbb{N} \times \mathbb{Z}$ grid minor.
Using our characterisation, we describe the structure of graphs without $\mathbb{Z} \times \mathbb{Z}$ grid minors in terms of tree-decompositions.
\end{abstract}

\maketitle

\section{Introduction}

In extremal graph theory it is common to analyse the structure of graphs which do not contain a certain minor or subdivision of some graph.
This goes hand in hand with the search for conditions in terms of graph invariants, such as degree conditions, that force the existence of certain minors or subdivisions.
Extending the scope of extremal questions to include infinite graphs, it is helpful to consider new graph invariants, which may not be defined for finite graphs, in order to gain more information about the structure of infinite graphs.
For an overview of results in the field of extremal infinite graph theory see the surveys of Diestel~\cite{diestel_arx} and of Stein~\cite{stein}.

One example for such a new invariant is the degree of an end of a graph.
The \textit{ends} of a graph are the equivalence classes of the \textit{rays}, i.e., one-way infinite paths, where we say that two rays are equivalent if an only if they cannot be separated by finitely many vertices in the graph.
Now the \textit{degree} of an end is defined as the maximum number of disjoint rays in this end (including `infinitely many').
The foundation of this definition, namely, that the end degree is well-defined, is provided by the following theorem of Halin.

\begin{theorem}\label{inf_rays}\cite[Satz~1]{halin}
If a graph contains $n$ pairwise disjoint rays for every $n \in \mathbb{N}$, then it contains infinitely many pairwise disjoint rays.
\end{theorem}

\noindent Furthermore, although without stating the term `end degree' explicitly, Halin used the following theorem to show that an end of infinite degree forces the existence of an $\mathbb{N} \times \mathbb{N}$ grid minor.
In fact he actually proved that it forces a subdivision of the graph $H^{\infty}$ shown in Figure~\ref{H_inf}.
Then the statement about the $\mathbb{N} \times \mathbb{N}$ grid minor follows, since the graph $H^{\infty}$ contains the $\mathbb{N} \times \mathbb{N}$ grid as a minor.
Note that the question of whether a graph contains an $\mathbb{N} \times \mathbb{Z}$ grid minor is not more difficult than asking for an $\mathbb{N} \times \mathbb{N}$ grid minor since the $\mathbb{N} \times \mathbb{N}$ grid contains a subdivision of the $\mathbb{N} \times \mathbb{Z}$ grid.

\begin{theorem}\label{NxN}\cite[Satz~$4'$]{halin}
Whenever a graph contains infinitely many pairwise disjoint and equivalent rays, it contains a subdivision of $H^\infty$.
\end{theorem}

\noindent Beside Halin's proof of Theorem~\ref{NxN}, there is now also a shorter proof of this theorem by Diestel (see \cite{new_grid} or \cite[Thm.~8.2.6]{diestel_buch}).
Note that the converse of this implication is obviously true as well.
So Theorem~\ref{NxN} gives a characterisation of graphs without a subdivision of $H^\infty$ and therefore also of graphs without an $\mathbb{N} \times \mathbb{Z}$ grid minor.

Robertson, Seymour and Thomas characterized the structure of graphs without ${\mathbb{N} \times \mathbb{Z}}$ grid minors as those that have tree-decompositions into finite parts and with finite adhesion.
A tree-decomposition into finite parts has \textit{finite adhesion} if along each ray of the tree the sizes of the adhesion sets corresponding to its edges are infinitely often less than some fixed finite number.
Given such a tree-decomposition, an $\mathbb{N} \times \mathbb{Z}$ grid minor cannot be contained in a part because all of these are finite.
The only other possibility where such a grid minor could lie in a graph would be in the union of the parts along a ray of the tree of the tree-decomposition.
However, the finite adhesion prevents this possibility.

\begin{theorem}\label{NxZ_td}\cite[(2.6)]{rob_sey_thom}
A graph has no $\mathbb{N} \times \mathbb{Z}$ grid minor if and only if it has a tree-decomposition into finite parts and with finite adhesion.
\end{theorem}

While all the above theorems give characterisations for when graphs do or do not contain an ${\mathbb{N} \times \mathbb{Z}}$ grid minor, it was not clear whether a similar characterisation exists for $\mathbb{Z} \times \mathbb{Z}$ grids.
The main theorem in this paper, Theorem~\ref{main_thm}, and Corollary~\ref{cor_main} give characterisations for a $\mathbb{Z} \times \mathbb{Z}$ grid minor in the same spirit as the results above do for an $\mathbb{N} \times \mathbb{Z}$ grid minor.
The key idea is to consider not just sets of disjoint equivalent rays but \textit{bundles}, which are sets of disjoint equivalent rays having the additional property that there are infinitely many disjoint cycles that intersect with each ray of the bundle, but only in a path.
Graphically, the cycles of a bundle can be viewed as concentric cycles around the common end in which the rays of the bundle lie.
It is not difficult to see that graphs with a $\mathbb{Z} \times \mathbb{Z}$ grid minor contain arbitrarily large bundles.
But it turns out that the converse is also true, and so containing arbitrarily large bundles is not only necessary for the existence of a $\mathbb{Z} \times \mathbb{Z}$ grid minor, but also sufficient.
Now let us state the main theorem and its corollary precisely.
See Section~2 for the definitions of the involved terms.

\begin{theorem}\label{main_thm}
For a graph $G$ the following are equivalent:
\begin{enumerate}[\normalfont(i)]
\item There is an end $\omega$ of $G$ and $n$-bundles $B_n$ for every $n \in \mathbb{N}$ with $B_n \subseteq \omega$.
\item There is an $\infty$-bundle in $G$.
\item There is a consistent $\infty$-bundle in $G$.
\item $G$ contains a subdivision of the Dartboard.
\item $G$ contains a $\mathbb{Z} \times \mathbb{Z}$ grid as a minor.
\item $G$ contains a set $\mathcal{R}$ of infinitely many equivalent disjoint rays such that for every $R \in \mathcal{R}$ all rays in $\mathcal{R} \setminus \lbrace R \rbrace$ are still equivalent in $G-R$.
\end{enumerate}
\end{theorem}

\begin{corollary}\label{cor_main}
A graph has no $\mathbb{Z} \times \mathbb{Z}$ grid minor if and only if it has a bundle-narrow tree-decomposition into finite parts distinguishing all ends.
\end{corollary}

The rest of the paper is organized as follows.
In Section~2 we state the definitions and notation that we need in this paper.
Furthermore, we collect known results which we shall use in the proof of the main theorem and its corollary.
The proofs of Theorem~\ref{main_thm} and of Corollary~\ref{cor_main} are the content of Section~3.

\section{Preliminaries}

In this section, we list important definitions, notation and already known results needed for the rest of the paper.
In general, we will use the graph theoretical notation of \cite{diestel_buch} in this paper.
For basic facts about graph theory, especially for infinite graphs, the reader is referred to \cite{diestel_buch} as well.

All graphs we consider in this paper are undirected and simple.
Furthermore, we do not assume a graph to be finite unless we state this explicitly.

For $n \geq 3$ we write $C_{n}$ for the cycle with $n$ vertices and for $m, k \in \mathbb{N}$ we denote by $K_{m, k}$ the complete bipartite graph with $m$ vertices in one class and $k$ in the other.

We define the $\mathbb{N} \times \mathbb{N}$ grid as the graph whose vertex set is $\mathbb{N} \times \mathbb{N}$ and two vertices are adjacent if and only if they differ in only one component by precisely $1$.
The $\mathbb{Z} \times \mathbb{Z}$ grid and the $\mathbb{N} \times \mathbb{Z}$ grid are defined in the same way but with vertex set $\mathbb{Z} \times \mathbb{Z}$ or $\mathbb{N} \times \mathbb{Z}$, respectively, instead of $\mathbb{N} \times \mathbb{N}$.

The graph $H^\infty$ (see Fig.~\ref{H_inf}) is the graph obtained in the following way:
First take the $\mathbb{N} \times \mathbb{N}$ grid and delete the vertex $(0, 0)$ together with all vertices $(n, m)$ with $n > m$.
Furthermore, delete all edges $(n, m)(n+1, m)$ when $n$ and $m$ have equal parity.

Now let us make some remarks on the graph $H^\infty$.
It follows from the definition of $H^\infty$ that it is a subgraph of the $\mathbb{N} \times \mathbb{N}$ grid.
However, $H^\infty$ is still rich enough to contain the $\mathbb{N} \times \mathbb{N}$ grid as a minor.
This fact is not so hard to prove and we omit a proof of it.
Furthermore, every vertex in $H^\infty$ has either degree $2$ or $3$.
So having $H^\infty$ as a minor in a graph is equivalent to containing a subdivision of it.
So we can conclude that a graph has the $\mathbb{N} \times \mathbb{N}$ grid as a minor if and only if it contains a subdivision of $H^\infty$.

\begin{figure}[htbp]
\centering
\includegraphics{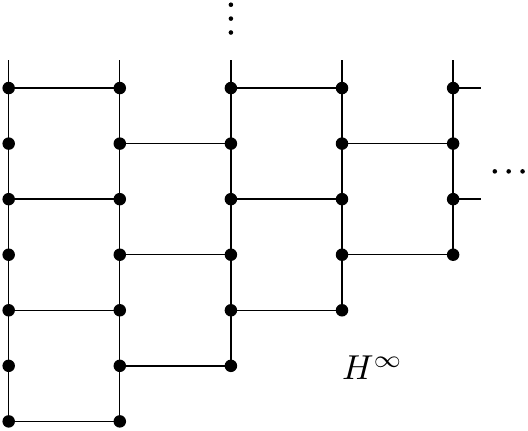}
\caption{The graph $H^\infty$.}
\label{H_inf}
\end{figure}

A one-way infinite path in a graph $G$ is called a \textit{ray} of $G$.
An equivalence relation can be defined on the set of all rays of $G$ by saying that two rays in $G$ are equivalent if they cannot be separated by finitely many vertices.
It is straightforward to check that this relation really defines an equivalence relation.
The corresponding equivalence classes of rays with respect to this relation are called the \textit{ends} of $G$.

A ray which is contained in an end $\omega$ of the graph is called an $\omega$-\textit{ray}.
The vertex of degree $1$ in a ray is called the \textit{startvertex} of the ray.
A subgraph of a ray $R$ which is itself a ray is called a \textit{tail} of $R$.

For ${n \in \mathbb{N}}$ a set of $n$ disjoint rays is called an $n$-\textit{bundle} if there are infinitely many disjoint cycles each of which intersects with each ray, but only in a path.
For every $n$-bundle, the cycles which witness that the $n$ disjoint rays are an $n$-bundle can be chosen in such a way that they all run through the rays in the same cyclic order.
We call such a set of cycles the \textit{embracing cycles} of the $n$-bundle.
Note that the definition of an $n$-bundle implies that an $n$-bundle is always a subset of one end.
For the rest of the paper, we will implicitly assume by stating that ${R_1, \ldots, R_n}$ are the rays of an $n$-bundle that the embracing cycles traverse them in order $R_1, \ldots, R_{n-1}, R_{n}$.

An infinite set of disjoint rays $\lbrace R_1, R_2, \ldots \rbrace$ is called an $\infty$-\textit{bundle} if there are disjoint cycles $C_i$ and natural numbers $c_i$ for every $i \in \mathbb{N}$ such that for all $i, j \in \mathbb{N}$ with $i < j$ we have $c_i < c_j$ and $C_i$ intersects with each $R_{\ell}$ for $\ell \leq c_i$, but only in a path.
Furthermore, we call an $\infty$-bundle \textit{consistent} if for all $i, j \in \mathbb{N}$ with $i < j$ the cycles $C_i$ and $C_j$ run through the rays $R_1, \ldots, R_{c_i}$ in the same cyclic order.
As for $n$-bundles we call the cycles $C_i$ embracing cycles.
Also note that the rays of an $\infty$-bundle are in the same end.

Now consider an $n$-bundle with rays ${R_1, \ldots, R_n}$ and a $k$-bundle whose rays are ${R'_1, \ldots, R'_k}$ where $n \leq k$.
We say that the $n$-bundle can be \textit{joined} to the $k$-bundle if there are vertices ${r_i \in V(R_i)}$ for every ${i \in \lbrace 1, \ldots, n \rbrace}$ and ${r'_j \in V(R'_j)}$ for every ${j \in \lbrace 1, \ldots, k \rbrace}$ together with $n$ pairwise disjoint ${r_i}$--${r'_{\sigma(i)}}$ paths, for some injection ${\sigma : \lbrace 1, \ldots, n \rbrace \longrightarrow \lbrace 1, \ldots, k \rbrace}$, each of which intersects ${\bigcup_i R_ir_i \cup \bigcup_j r'_jR'_j}$ only in its endvertices.
The involved paths are called \textit{joining paths}.

Finally, we call an $n$-bundle \textit{infinitely} \textit{joined} to a $k$-bundle if for every finite vertex set $S$ of the graph the $n$-bundle can be joined to the $k$-bundle such that the joining paths do not intersect with $S$.

In order to define an archetypal example of a graph containing an $\infty$-bundle, we have to construct a sequence ${(G_i)_{i \in \mathbb{N}}}$ of graphs first.
For this we need, furthermore, the function ${f: \mathbb{N} \longrightarrow \mathbb{N}}$ which is defined as follows:
\[ {f(i) =
\begin{cases}
4, & \text{if } i = 1 \\
2^{i} \cdot 3, & \text{if } i \geq 2.
\end{cases}} \]
Now we state the recursive definition of the graphs $G_i$.
Let $G_1$ be a $C_{4}$.
Next suppose $G_i$ has already been defined.
The construction yields that there is a unique cycle $D_i$ in $G_i$ which is isomorphic to $C_{f(i)}$ and contains all vertices that have degree $2$ in $G_i$, of which there are ${g(i) = \frac{1}{3} \cdot f(i+1)}$ many.
Enumerate these vertices according to the cyclic order in which they appear on $D_i$.
Now we obtain $G_{i+1}$ by taking $G_{i}$ together with a disjoint copy of $C_{g(i)}$ whose vertices we enumerate according to the cyclic order of this cycle too, adding an edge between the $j$-th vertex of $D_i$ and the $j$-th of $C_{g(i)}$ for each $j$ and subdividing each edge of $C_{g(i)}$ twice.
Finally, we define the \textit{Dartboard} (see Fig.~\ref{dartboard}) as ${\bigcup_i G_i}$.

\begin{figure}[htbp]
\centering
\includegraphics{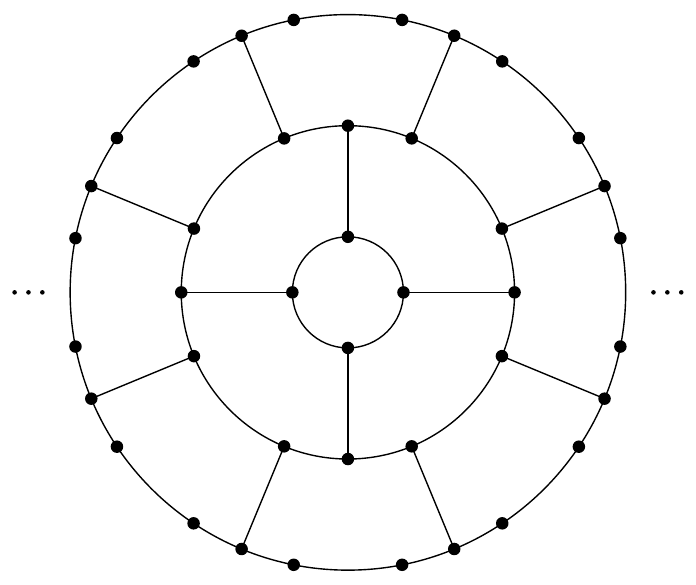}
\caption{The Dartboard.}
\label{dartboard}
\end{figure}

We continue with some remarks about normal spanning trees and tree-decom-
positions.
Let $T$ be a tree with root $r$ and let $t \in V(T)$.
Then we write $\lfloor t \rfloor$ for the up-closure of $t$ with respect to the tree-order of $T$ with root $r$.
Similarly, we write $\lceil t \rceil$ for the down-closure of $t$.

A rooted spanning tree of a graph is \textit{normal} if the endvertices of every edge in the graph are comparable in the tree-order.

The following theorem of Halin gives a very useful sufficient condition for the existence of a normal spanning tree.

\begin{theorem}\label{K_aleph_NST}\cite[Thm.~10.1]{halin_nst}
Every connected graph which does not contain a subdivision of a $K_{\aleph_0}$ has a normal spanning tree.
\end{theorem}

Next let us recall the definition of a tree-decomposition.
Let $G$ be a graph, $T$ be a tree and $(V_t)_{t \in V(T)}$ be a sequence of vertex sets of $G$.
We call $(T, (V_t)_{t \in V(T)})$ a \textit{tree-decomposition} of $G$ if the following three properties hold:

\begin{enumerate}
\item $V(G) = \bigcup_{t \in V(T)}V_t$.
\item For each edge $vw$ of $G$ there is a $t \in V(T)$ such that $v, w \in V_t$.
\item For all $t_1, t_2, t_3 \in V(T)$ such that $t_2$ lies on the unique $t_1$--$t_3$ path in $T$ the inclusion $V_{t_1} \cap V_{t_3} \subseteq V_{t_2}$ is true.
\end{enumerate}

We call a tree-decomposition $(T, (V_t)_{t \in V(T)})$ \textit{rooted} if the corresponding tree $T$ is rooted.
For a rooted tree-decomposition $(T, (V_t)_{t \in V(T)})$ whose tree $T$ has root $r$, we write $(T, r, (V_t)_{t \in V(T)})$.

A graph has a tree-decomposition \textit{into finite parts} if there is a tree-decomposition $(T, (V_t)_{t \in V(T)})$ of the graph with $V_t$ finite for every $t \in V(T)$.

We say that a tree-decomposition $(T, (V_t)_{t \in V(T)})$ of a graph $G$ into finite parts \textit{distinguishes} all ends of $G$ if for every ray $t_1 t_2 \ldots$ of $T$ all rays of $G$ that intersect all but finitely many $V_{t_i}$ are equivalent.
Since all parts of such a tree-decomposition are finite, there is an injection from the set of ends of $G$ to the set of ends of $T$.

An easy observation shows that we always get a tree-decomposition into finite parts distinguishing all ends as soon as we have a normal spanning tree.

\begin{lemma}\label{NST_gives_td}
Every graph with a normal spanning tree has a tree-decomposition into finite parts distinguishing all ends.
\end{lemma}

\begin{proof}
Let $T$ be a normal spanning tree of a graph $G$ with root $r$.
Then we define the desired tree-decomposition as $(T, r, (\lceil t \rceil)_{t \in V(T)})$.
Let us briefly check that this really defines a tree-decomposition.
It is obvious that each vertex $v$ lies in some part, for example in $\lceil v \rceil$.
Since $T$ is normal, we know that the endvertices of every edge are comparable and must therefore lie in some common part.
Note for the remaining property that for all $t_1, t_3 \in V(T)$ we have $\lceil t_1 \rceil \cap \lceil t_3 \rceil = \lceil t \rceil$ where $t$ is the greatest vertex in the tree-order which is still comparable with $t_1$ and $t_3$.
Since every vertex $t_2$ on the $t_1$--$t_3$ path in $T$ is greater than $t$, we get that $\lceil t_1 \rceil \cap \lceil t_3 \rceil = \lceil t \rceil \subseteq \lceil t_2 \rceil$.

The definition of $(T, r, (\lceil t \rceil)_{t \in V(T)})$ ensures that every part is finite.
So it remains to check that this tree-decomposition distinguishes all ends.
Let us fix a ray ${r = t_1 t_2 \ldots}$ of $T$ and suppose there are two rays in $G$ which intersect with all but finitely many parts $\lceil t_i \rceil$.
Since $\bigcup_{i \leq k} \lceil t_i \rceil \setminus \bigcup_{i \leq j} \lceil t_i \rceil$ always induces a connected subgraph for $k > j$, we get that the two rays cannot be separated by finitely many vertices, which means they are equivalent.
\end{proof}

A tree-decomposition $(T, (V_t)_{t \in V(T)})$ has \textit{finite adhesion} if for every $t \in V(T)$ there is an integer $n \geq 0$ such that $|V_s \cap V_t| \leq n$ for every $s$ being adjacent with $t$ in $T$ and additionally for every ray $t_1 t_2 \ldots$ of $T$ there is an integer $k$ such that $|V_{t_i} \cap V_{t_{i+1}}| \leq k$ holds for infinitely many $i \in \mathbb{N}$.

By Theorem~\ref{NxZ_td} a tree-decomposition of a graph $G$ into finite parts and with finite adhesion is a witness that $G$ does not contain an $\mathbb{N} \times \mathbb{Z}$ grid minor.
Beside the requirement that each part shall be too small to contain a grid minor, which is done by requiring all parts to be finite, the possibility to distribute a grid minor along a branch in the tree-decomposition is prevented by making all branches too narrow for arbitrarily many rays to run through them.
The latter goal is achieved by requiring the tree-decomposition to have finite adhesion.

setting the adhesion parameter of the tree-decomposition to be finite.

Similar to the definition before we now introduce a property that prevents from distributing a $\mathbb{Z} \times \mathbb{Z}$ grid minor along a whole branch in a tree-decomposition.
Unfortunately, verifying this property needs a closer look at the graph and the bundles in it, in contrast to the more abstract property of finite adhesion, which involves only the tree and parts of the decomposition.

A tree-decomposition $(T, (V_t)_{t \in V(T)})$ is called \textit{bundle-narrow} if for every ray $t_1 t_2 \ldots$ of $T$ there is an integer $k \geq 1$ such that there is no $k$-bundle in $G$ whose rays intersect all but finitely many $V_{t_i}$.

We close this section with a well-known result about $2$-connected graphs. We will need this lemma in the proof of the main theorem.

\begin{lemma}\label{C_k_or_K_2,k}\cite[Prop.~9.4.2]{diestel_buch}
For every positive integer $k$, there exists an integer $n$ such that every $2$-connected graph on at least $n$ vertices contains a subgraph isomorphic to a subdivision of either $K_{2, k}$ or a cycle of length $k$.
\end{lemma}

\section{Proof of the main theorem}

Before we can prove Theorem~\ref{main_thm} we have to make some observations about bundles.
We start with the following lemma which tells us in our context of bundles that we can join a bundle to another one which is sufficiently large as soon as both are subsets of the same end.

\begin{lemma}\label{ray_linking_lemma}
Let $G$ be a graph, $\omega$ be an end of $G$ and $k \geq n \geq 1$ be integers.
Furthermore, let ${\mathcal{R} = \lbrace R_1, \ldots, R_n \rbrace}$ and ${\mathcal{R}' = \lbrace R'_1, \ldots, R'_k \rbrace}$ be sets of $n$ and $k$ pairwise disjoint $\omega$-rays, respectively.
Then there are vertices ${r'_i \in V(R'_i)}$ for each $i$ with ${1 \leq i \leq k}$ such that there are $n$ pairwise disjoint paths between the start vertices of the rays in $\mathcal{R}$ and the vertices ${r'_1, \ldots, r'_k}$ each of which intersects ${\bigcup_i r'_iR'_i}$ at most in $r'_i$.
\end{lemma}

\begin{proof}
We want to work within a finite subgraph $H$ of $G$ in which we find the desired paths.
To define $H$ we take a set $\mathcal{P}$ of $kn^2$ pairwise disjoint paths such that for every ${i \in \lbrace 1, \ldots, n \rbrace}$ and every ${j \in \lbrace 1, \ldots, k \rbrace}$ there are $n$ disjoint $R_{i}$--$R'_{j}$ paths in $\mathcal{P}$.
This is possible since all rays lie in the same end.
For all $i, j$ with ${1 \leq i \leq n}$ and ${1 \leq j \leq k}$ let $r_i$ be the last vertex on $R_i$ which is an endvertex of one of the $kn$ many $R_i$--$R'_j$ paths from $\mathcal{P}$ and $r'_j$ be the last vertex on the ray $R'_j$ which is hit by any path from $\mathcal{P}$ or any $R_ir_i$.
Next we define $H$ as follows:
\[ H:= G \Big[ \bigcup^n_{i=1} V(R_{i}r_{i}) \cup V \Big( \bigcup \mathcal{P} \Big) \cup \bigcup^k_{j=1} V(R'_{j}r'_{j}) \Big]. \]

We complete the proof of this lemma by showing that there are $n$ disjoint paths from the start vertices of the rays in $\mathcal{R}$ to $n$ vertices of the set ${\lbrace r'_1, \ldots, r'_k \rbrace}$ in the graph $H$. 
By Menger's Theorem it is sufficient to prove that there is no set $S$ of less than $n$ vertices which separates the start vertices of the rays in $\mathcal{R}$ from the vertices ${r'_1, \ldots, r'_k}$ in $H$.
Suppose for a contradiction that such a set $S$ exists in $H$.
Since $S$ contains less than $n$ vertices and the paths $R_{i}r_{i}$ are pairwise disjoint, we can find an index $\ell$ such that $R_{\ell}r_{\ell}$ does not contain any vertex of $S$.
The same is true for the paths $R'_{i}r'_{i}$ with some index $p$.
Furthermore, we can find an $R_{\ell}$--$R'_{p}$ path ${P_{\ell p} \in \mathcal{P}}$ that is disjoint from $S$ because $\mathcal{P}$ contains $n$ many $R_{\ell}$--$R'_{p}$ paths.
Now we have a contradiction because the union of the paths $R_{\ell}r_{\ell}$, $P_{\ell p}$ and $R'_{p}r'_{p}$ contains a path from the startvertex of the ray $R_{\ell}$ to $r'_{p}$ that avoids $S$.
\end{proof}

By iterating Lemma~\ref{ray_linking_lemma} and using the fact that there are only finitely many injections which correspond to path systems of joining paths from an $n$-bundle to a $k$-bundle for $k \geq n$, we obtain the following corollary.

\begin{corollary}\label{inf_joined}
Let $G$ be a graph and $\omega$ be an end of $G$.
Then an $n$-bundle $B_n$ is infinitely joined to a $k$-bundle $B_k$ if $k \geq n$ and $B_n, B_k \subseteq \omega$.
\end{corollary}

\begin{proof}
First we apply Lemma~\ref{ray_linking_lemma} to the rays of $B_n$, say $\lbrace R_1, \ldots, R_n \rbrace$, and $B_k$, say $\lbrace R'_1, \ldots, R'_k \rbrace$.
Let $\mathcal{P}_1$ be the resulting path system.
Next we delete the finite subgraph $H$ of $G$ defined as in the proof of Lemma~\ref{ray_linking_lemma} from $G$.
By the definition of bundles, the tails of $B_n$ and $B_k$ in $G-H$ are still bundles and all of these tails are still equivalent.
Next we apply Lemma~\ref{ray_linking_lemma} to these tails and obtain a path system $\mathcal{P}_2$.
By iterating this argument, we get path systems $\mathcal{P}_i$ for $i \in \mathbb{N}$ such that $P \cap Q = \emptyset$ for every $P, Q \in \bigcup_{i \in \mathbb{N}} \mathcal{P}_i$ with $P \neq Q$ and each path system $\mathcal{P}_i$ connects the $n$ rays of $B_n$ with $n$ distinct rays of $B_k$.
Since there is only a finite bounded number of possibilities on which rays the start- and endvertices of the paths of some path system $\mathcal{P}_i$ can be, we obtain by the pigeon hole principle that there is an infinite subset $\lbrace \mathcal{P}'_j \; ; \; j \in \mathbb{N} \rbrace \subseteq \lbrace \mathcal{P}_i \; ; \; i \in \mathbb{N} \rbrace$ of path systems and an injection $\sigma: \lbrace 1, \ldots, n \rbrace \longrightarrow \lbrace 1, \ldots, k \rbrace$ such that each path system $\mathcal{P}'_j$ contains a path from $R_i$ to $R'_{\sigma(i)}$ for all $i \in \lbrace 1, \ldots, n \rbrace$.
So the set $\lbrace \mathcal{P}'_j \; ; \; j \in \mathbb{N} \rbrace$ of disjoint path systems witnesses that $B_n$ is infinitely joined to $B_k$.
\end{proof}

For $n$-bundles it follows from the pigeon hole principle that we can always find an infinite subset of the embracing cycles whose elements induce the same cyclic order on the rays of the $n$-bundle.
So without loss of generality we could assume that the embracing cycles of an $n$-bundle run through the rays of the bundle always in the same cyclic order.
We can do a similar thing for $\infty$-bundles, but it involves an application of the compactness principle rather than the pigeon hole principle.
So before we make the corresponding statement about $\infty$-bundles precise, let us state a version of the compactness principle we will make use of, namely K\"onig's Lemma:

\begin{lemma}\label{Koenig}\cite[Lemma~8.1.2]{diestel_buch}
Let $V_0, V_1, \ldots$ be an infinite sequence of disjoint non-empty finite sets, and let $G$ be a graph on their union. Assume that every vertex in a set $V_n$ with $n \geq 1$ has a neighbour in $V_{n-1}$. Then $G$ contains a ray $v_0v_1 \ldots$ with $v_n \in V_n$ for all $n$.
\end{lemma}

Now the next lemma tells us that we always obtain a consistent $\infty$-bundle from an $\infty$-bundle.

\begin{lemma}\label{cons_inf-bundle}
The rays of an $\infty$-bundle $B_{\infty}$ form also a consistent $\infty$-bundle witnessed by an infinite subset of the embracing cycles of $B_{\infty}$.
\end{lemma}

\begin{proof}
Let $B_{\infty} = \lbrace R_1, R_2, \ldots \rbrace$ be an $\infty$-bundle of a graph and let $\lbrace C_i \; ; \; i \in \mathbb{N} \rbrace$ be the set of its embracing cycles.
Furthermore, let the natural numbers $c_i$ be given as in the definition of an $\infty$-bundle.
Now we define an auxiliary graph $G$ to apply K\"onig's Lemma.
For every $n \geq 1$ let $V_n \subseteq V(G)$ be the set of all cyclic orders of how an embracing cycle $C_j$ runs through the set of rays $\lbrace R_1, \ldots, R_{c_n} \rbrace$ for $j \geq n$. So each set $V_n$ is finite and non-empty.
Furthermore, let there be an edge in $G$ between vertices $v_n \in V_{n}$ and $v_{n+1} \in V_{n+1}$ if the cyclic order $v_{n+1}$ restricted to the set $\lbrace R_1, \ldots, R_{c_n} \rbrace$ is equal to $v_n$.
With these definitions all requirements for K\"onig's Lemma (Lemma~\ref{Koenig}) are fulfilled.
So $G$ contains a ray $v_1v_2 \ldots$ with $v_n \in V_n$ for every $n \geq 1$.
This allows us to take cycles $C_{k_n}$ such that $C_{k_n}$ induces $v_n$ on the rays $\lbrace R_1, \ldots, R_{c_n} \rbrace$ for every $n \geq 1$ where $k_n > k_{n'}$ holds for $n > n'$.
These cycles witness that $B_{\infty}$ is a consistent $\infty$-bundle.
\end{proof}

Now we are prepared to prove Theorem~\ref{main_thm}.

\begin{proof}[Proof of Theorem~\ref{main_thm}]
The implication from (ii) to (iii) is true by Lemma~\ref{cons_inf-bundle}.

Showing that (iv) follows from (iii) is not difficult.
We sketch the proof of this implication.
Construct subdivisions of the defining subgraphs $G_i$ of the Dartboard inductively.
Start with an embracing cycle of the consistent $\infty$-bundle that runs through $f(1) = 4$ rays of the $\infty$-bundle as $G_1$.
Now suppose we have already constructed a subdivision $H_n$ of $G_n$ and there are $f(n)$ tails $T_1, T_2, \ldots,  T_{f(n)}$ of rays of the $\infty$-bundle that intersect with $H_n$ only in their startvertices.
Pick $f(n)$ many embracing cycles $C'_1, C'_2, \ldots, C'_{f(n)}$ of the $\infty$-bundle that are disjoint from $H_n$, each traversing the tails $T_1, T_2, \ldots T_{f(n)}$, and another embracing cycle $C$ which is disjoint from $H_n$, comes later in the enumeration of all embracing cycles than the ones we have picked so far and traverses at least $f(n+1)$ many rays of the $\infty$-bundle including the $f(n)$ tails $T_i$.
Since the $\infty$-bundle is consistent, we can use the cycles $C, C'_1, \ldots, C'_{f(n)}$ and the tails $T_i$ to find a subdivision of $H_{n+1}$ together with $f(n+1)$ many tails of rays of the $\infty$-bundle that intersect with $H_{n+1}$ only in their startvertices.
In this step we possibly have to reroute some of the tails $T_i$ using the cycles $C'_i$ in order to get compatible paths from $H_n$ to $C \subseteq H_{n+1}$.
Using this construction the union $\bigcup_n H_n$ gives us a subdivision of the Dartboard.

The implications from (iv) to (v) and from (v) to (vi) are easy and so we omit the details.

Now we look at the implication from (i) to (ii).
Let $\omega$ be an end of a graph $G$ such that there are $n$-bundles $B_n = \lbrace R^n_1, \ldots, R^n_n \rbrace \subseteq \omega$ for every $n \in \mathbb{N}$.
We construct an $\infty$-bundle inductively.
In step $i$ we shall have a graph $H_i$ which satisfies the following properties:

\begin{enumerate}
\item $H_i$ is the union of disjoint cycles $C_1, \ldots, C_i$ and disjoint paths $P^i_1, \ldots, P^i_i$.
\item The intersection of $P^i_j$ with $C_k$ is a path for all $j, k$ with $j \leq k \leq i$.
\item $P^i_j \cap C_k = \emptyset$ holds for all $j, k$ with $k < j \leq i$.
\item Each path $P^i_j$ runs through the cycles $C_j, \ldots, C_i$ in the order of their enumeration.
\item $H_i \cap H_{i-1} = H_{i-1}$ for $1 < i$.
\item $P^{i-1}_j$ is an initial segment of $P^i_j$ for every $j \leq i$ with $1 < i$.
\item In $G$ there exist tails of rays of some $n$-bundle $B_n$ such that every endvertex of a path $P^i_j$ in $H_i - H_{i-1}$ with $j \leq i$ is a startvertex of one of these tails but apart from that the tails are disjoint from $H_i$.
\end{enumerate}

For $H_1$ we take an embracing cycle of $B_1$ as $C_1$ and set $H_1 = C_1$.
We define $P^1_1$ to be the trivial path which is the last vertex $v$ of $R^1_1$ on $C_1$.
So (1), (3), (4), (5) and (6) are obviously satisfied.
Property (2) holds by the definition of embracing cycle.
For (7) we can take the tail $vR^1_1$ of $R^1_1$.

Now suppose we have already defined $H_i$ which satisfies the seven stated properties.
Let $B_n$ be the $n$-bundle which we get from property (7) for step $i$.
By Corollary~\ref{inf_joined} we get that $B_n$ is infinitely joined to any $k$-bundle $B_k$ if $k \geq n$.
Let us fix an integer $k$ with $k > n \geq i$.
Since $H_i$ is a finite graph, we can find joining paths $Q_1, \ldots, Q_n$ from $B_n$ to $B_k$ which meet $H_i$ only in the endvertices of the paths $P^i_j$.
Now fix an embracing cycle $C$ of $B_k$ that is disjoint from $H_i$ such that the tails of the rays of $B_k$ starting from $C$ are disjoint from $H_i$ as well as from the joining paths $Q_j$.
We set $C_{i+1} = C$.
Furthermore, we define $P^{i+1}_j$ for $j \leq i$ to be the concatenation of $P^{i}_j$ with the joining path $Q_{j'}$ which it intersects and with the path $Q^C_{j'}$ where $Q^C_{j'}$ is the path which starts at the endvertex of $Q_{j'}$ which  lies on a ray of $B_k$ and follows that ray up to the last vertex that is in the intersection of this ray with $C_{i+1}$.
Since $k > i$ holds, there is a ray $R^*$ in $B_k$ whose tail with startvertex in $C_{i+1}$ does not intersect with any of the paths $P^{i+1}_j$.
Now we set $P^{i+1}_{i+1}$ to be the trivial path consisting of the last vertex on $R^*$ which lies also on $C_{i+1}$.
Finally, we set $H_{i+1}$ to be the union of $H_i$ with all paths $P^{i+1}_j$.
It remains to check that the definitions we made for step $i+1$ ensure that the properties (1) to (7) are also true for $H_{i+1}$.
Property (1), (5) and (6) are obviously true by definition.
Since (5) and (6) hold and the paths $Q_{j}$ and $Q^C_{j}$ are chosen to be disjoint from $H_i$ except for one starting vertex of each $Q_{j}$, we need to check property (2) just for the paths $P^{i+1}_j$ and the cycle $C_{i+1}$.
Note that the intersection of a path $P^{i+1}_j$ with the cycle $C_{i+1}$ is equal to the intersection of one of the rays of $B_k$ with $C_{i+1}$.
So this intersection is just a path because $C_{i+1}$ is an embracing cycle of $B_k$.
Property (3) and (4) are valid because of property (2) and since $P^{i+1}_j - P^{i}_j$ is disjoint from $H_i$.
The bundle $B_k$ together with suitable tails of its rays starting in $C_{i+1}$ we chose in the construction for step $i+1$ witnesses that property (7) holds.

Using the sequence of graphs $(H_i)_{i \in \mathbb{N}}$, we are able to define an $\infty$-bundle $B_{\infty}$.
We set $R^{\infty}_j = \bigcup_{i \in \mathbb{N}}P^i_j$ for every $j \in \mathbb{N}$ and then $B_{\infty} = \lbrace R^{\infty}_j \; ; \; j \in \mathbb{N} \rbrace$.
Property (6) ensures that each $R^{\infty}_j$ is a ray and the disjoint cycles $C_i$ together with property (2) ensure that $B_{\infty}$ is indeed an $\infty$-bundle.
This completes the proof that (i) implies~(ii).

It remains to prove the implication from (vi) to (i).
Let $\omega$ be the end of $G$ which contains $\mathcal{R}$ as a subset.
Next let us fix an arbitrary $k \in \mathbb{N}$ and show that there is a $k$-bundle in the graph $G$ all whose rays are elements of $\omega$.
For this purpose we choose $n$ disjoint rays $R_1, \ldots, R_n$ from the set $\mathcal{R}$ where $n$ is as big as the integer $n$ from Lemma~\ref{C_k_or_K_2,k} with our fixed $k$ as input.
Next we define an auxiliary graph $H$ to which we shall apply that lemma.
First set $V(H) = \lbrace R_1, \ldots, R_n \rbrace$.
Furthermore, we say that there is an edge $R_iR_j$ if and only if there exist infinitely many disjoint $R_i$--$R_j$ paths in $G$ which are disjoint from all rays in $\lbrace R_1, \ldots , R_n \rbrace \setminus \lbrace R_i, R_j \rbrace$.
In order to apply Lemma~\ref{C_k_or_K_2,k} to $H$, we need to check that $H$ is $2$-connected.
Suppose for a contradiction that there exists a ray $R_{\ell}$ such that $H-R_{\ell}$ is not connected.
So we can find a bipartition $(A, B)$ of $V(H) \setminus \lbrace R_{\ell} \rbrace$ which yields an empty cut of $H$.
Now let us fix a ray $R \in A$ and $R' \in B$.
We know by assumption that $R$ and $R'$ are equivalent in $G - R_{\ell}$.
This implies that there are infinitely many disjoint $R$--$R'$ paths in $G - R_{\ell}$.
Using the pigeon hole principle and the fact that $A$ and $B$ contain less than $n$ rays, infinitely many of these paths have a common last ray of $A$ and a common first ray in $B$ which they intersect, but this tells us that there exists an $A$--$B$ edge in $H-R_{\ell}$.
So we have a contradiction and can conclude that $H$ is $2$-connected.
Now we apply Lemma~\ref{C_k_or_K_2,k} to $H$.
If the lemma tells us that $H$ contains a subdivided cycle of length at least $k$, then we immediately also get a $k$-bundle in $G$ all whose rays are elements of $\omega$.
So suppose there is a subdivision of $K_{2, k}$ in $H$.
Without loss of generality let $R^*_1$, $R^*_2$ and $R_1, \ldots, R_k$ be branch vertices of the subdivided $K_{2, k}$ in $H$ such that there are disjoint paths from $R^*_1$ and $R^*_2$ to $R_i$ for every $i$ with $1 \leq i \leq k$ in $H$.
Now we use the rays $R^*_1$ and $R^*_2$ as distributing rays in $G$ to build infinitely many disjoint cycles that witness $\lbrace R_1, \ldots, R_k \rbrace$ being a $k$-bundle.
The cycles can be built all in the same way:
First pick a $R_1$--$R^*_2$ path $P^*_1$ which is disjoint from $R^*_1$ and from each ray $R_i$ for $1 \leq i \leq k$ and $i \neq 1$. Now start at the endvertex of $P^*_1$ on $R_1$ and follow that ray until there is a $R_1$--$R^*_1$ path $P_1$ which is disjoint from $R^*_2$, $P^*_1$ and from each ray $R_i$ for $1 \leq i \leq k$ and $i \neq 1$.
Then follow $P_1$ and $R^*_1$ afterwards until there is a $R^*_1$--$R_2$ path which is disjoint from $R^*_2$, $P^*_1$, $P_1$ and from each ray $R_i$ for $1 \leq i \leq k$ and $i \neq 2$.
Repeating this pattern we get a $R^*_2$--$R_k$ path $Q$ which meets every ray $R_i$ for $1 \leq i \leq k$ only in a path.
Then we can close $Q$ to obtain a cycle by following $R_k$ from the endvertex of $Q$ on $R_k$ until there is a $R_k$--$R^*_2$ path $P^*_2$ that is disjoint from $R^*_1$, from each ray $R_i$ for $1 \leq i \leq k$ and $i \neq k$ and from each path we have used so far, then following $P^*_2$ and finally using the $P^*_2$--$P^*_1$ path on $R^*_2$.
By deleting large enough initial segments from all rays, we can repeat the construction of such cycles infinitely often and obtain the desired sequence of disjoint cycles witnessing that $\lbrace R_1, \ldots, R_k \rbrace \subseteq \omega$ is a $k$-bundle.
\end{proof}

Using Theorem~\ref{main_thm} we prove now Corollary~\ref{cor_main}, which describes the structure of graphs without $\mathbb{Z} \times \mathbb{Z}$ grid minor in terms of bundle-narrow tree-decompositions.

\begin{proof}[Proof of Corollary~\ref{cor_main}]
Let $G$ be a graph and let us assume that it does not contain a $\mathbb{Z} \times \mathbb{Z}$ grid minor.
So $G$ cannot contain a subdivision of $K_{\aleph_0}$ either and we can apply Theorem~\ref{K_aleph_NST} telling us that $G$ has a normal spanning tree.
Using Lemma~\ref{NST_gives_td} we obtain a tree-decomposition of $G$ into finite parts distinguishing all ends.
Now we know that for every ray $t_1 t_2 \ldots$ of $T$ all rays of $G$ that intersect all but finitely many of the parts $V_{t_i}$ are equivalent in $G$.
Using the equivalence of (i) and (v) in Theorem~\ref{main_thm}, we can furthermore find for each end of $G$ the least integer $k \geq 1$ such that no $k$-bundle exists in this end.
Combining these two observations, we can find for every ray $t_1 t_2 \ldots$ of $T$ the least integer $k \geq 1$ such that there is no $k$-bundle in $G$ whose rays intersect with all but finitely many of the parts $V_{t_i}$.
So our tree-decomposition of $G$ into finite parts which distinguishes all ends is already bundle-narrow.

For the other direction let us assume that a graph graph $G$ has a $\mathbb{Z} \times \mathbb{Z}$ grid minor and suppose for a contradiction that it also has a bundle-narrow tree-decomposition $(T, (V_t)_{t \in V(T)})$ into finite parts.
Using that all parts $V_t$ are finite, we can look at the last time a ray $R$ of $G$ leaves a part $V_t$.
In this way $R$ induces a ray $t_1 t_2 \ldots$ of $T$ such that $R$ intersects each part $V_{t_i}$.
Note that equivalent rays in $G$ induce rays in $T$ which have a common tail, because they cannot be separated by finitely many vertices in $G$.
By the equivalence of (i) and (v) in Theorem~\ref{main_thm}, there exists an end of $G$ which contains $n$-bundles for every $n \in \mathbb{N}$.
We know that the rays of all these bundles induce rays of $T$ that lie in the same end of $T$.
Now any ray of $T$ that belongs to this end of $T$ contradicts our assumption that the tree-decomposition is bundle-narrow.
\end{proof}

\pagebreak

\end{document}